\documentclass[11pt]{amsart}
\usepackage{latexsym}
\usepackage{amsxtra}
\usepackage{amssymb}
\usepackage{amsthm} 
\usepackage{xspace}

\usepackage{euscript}
\input xy
\xyoption{all} \CompileMatrices

\usepackage{geometry}
\numberwithin{equation}{section}

\newtheorem{theorem}{Theorem}[section]
\newtheorem{proposition}[theorem]{Proposition}
\newtheorem{corollary}[theorem]{Corollary}
\newtheorem{conjecture}[theorem]{Conjecture}
\newtheorem{lemma}[theorem]{Lemma}

\theoremstyle{definition}
\newtheorem{remark}[theorem]{Remark}
\newtheorem{example}[theorem]{Example}
\newtheorem{problem}[theorem]{Problem}

\begin{document}

\def\XX{{\bf{X}}}
\def\YY{{\bf{Y}}}
\def\EE{{\bf{E}}}

\def\reg{\operatorname{reg}}
\def\ann{\operatorname{ann}}
\def\Spec{\operatorname{Spec}}
\def\Cl{\operatorname{Cl}}
\def\Cld{\operatorname{-Cl}}
\def\Inv{\operatorname{Inv}}
\def\Int{\operatorname{Int}}
\def\Prin{\operatorname{Prin}}
\def\Pic{\operatorname{Pic}}
\def\F{\operatorname{F}}
\def\SS{\operatorname{S}}
\def\SSS{\operatorname{S}}
\def\Inver{\operatorname{-Inv}}
\def\Hom{\operatorname{Hom}}
\def\TM{t\textup{-Max}}
\def\TS{t\textup{-Spec}}
\def\Max_\ast{\ast\textup{-Max}}
\def\Max{\textup{Max}}
\def\SSpec{\textup{-Spec}}
\def\SGV{\ast\textup{-GV}}
\def\WBF{\operatorname{WBF}}
\def\sKr{\operatorname{sKr}}
\def\Loc{\operatorname{Loc}}
\def\Ass{\operatorname{Ass}}
\def\wAss{\operatorname{wAss}}
\def\ann{\operatorname{ann}}
\def\Mod{\operatorname{mod}}
\def\ZZ{{\mathbb Z}}
\def\CC{{\mathbb C}}
\def\NN{{\mathbb N}}
\def\RR{{\mathbb R}}
\def\QQ{{\mathbb Q}}
\def\FF{{\mathbb F}}
\def\AA{{\mathbb A}}
\def\OO{{\mathcal O}}
\def\HH{{\mathbb{D}}}
\def\aa{{\bf{a}}}
\def\bb{{\bf{b}}}
\def\cc{{\bf{c}}}
\def\nn{{\bf{n}}}
\def\qq{{\bf{q}}}
\def\ss{{\bf{s}}}
\def\tt{{\bf{t}}}
\def\uu{{\bf{u}}}
\def\xx{{\bf{x}}}
\def\aaa{{\mathfrak a}}
\def\bbb{{\mathfrak b}}
\def\ppp{{\mathfrak p}}
\def\qqq{{\mathfrak q}}
\def\mm{{\mathfrak m}}
\def\PPP{{\mathfrak P}}
\def\nnn{{\underline{n}}}
\def\mmm{{\underline{m}}}
\def\lll{{\underline{l}}}
\def\kkk{{\underline{k}}}
\def\coeff{{\delta}}
\def\coeffb{{d}}
\def\ann{\operatorname{ann}}

\title{Integer-valued polynomials on commutative rings and modules}
\date{\today} \author{Jesse Elliott} \address{Department of Mathematics\\ California
State University, Channel Islands\\ Camarillo, California 93012}
\email{jesse.elliott@csuci.edu}

\maketitle

\begin{abstract}
The ring of integer-valued polynomials on an arbitrary integral domain is well-studied.  In this paper we initiate and provide motivation for the study of integer-valued polynomials on commutative rings and modules.  Several examples are computed, including the integer-valued polynomials over the ring $R[T_1,\ldots, T_n]/(T_1(T_1-r_1), \ldots, T_n(T_n-r_n))$ for any  commutative ring $R$ and any elements $r_1, \ldots, r_n$ of $R$,  as well as the integer-valued polynomials over the Nagata idealization $R(+)M$ of $M$ over $R$, where $M$ is an $R$-module such that every non-zerodivisor on $M$ is a non-zerodivisor of $R$.
\end{abstract}

\section{Introduction}

The ring of integer-valued polynomials on an integral domain is well-studied \cite{cah} \cite{nar}.  In this paper we initiate and provide motivation for the study of integer-valued polynomials on commutative rings and modules.

All rings in this paper are assumed commutative with identity.   Let $R$ be a ring.  The {\it total quotient ring} $T(R)$ of $R$ is the localization $U^{-1}R$ of the ring $R$ at the multiplicative set $U = R^{\reg}$ of all non-zerodivisors of $R$.  We define
$$\Int(R) = \{f \in T(R)[X]: f(R) \subseteq R\},$$
which is a subring of $T(R)[X]$ containing $R[X]$.  We call the ring $\Int(R)$ the ring of {\it integer-valued polynomials on $R$}.   The most well-studied case is where $R = \ZZ$: the ring $\Int(\ZZ)$ is  called  the ring of {\it integer-valued polynomials} and is known to be a non-Noetherian Pr\"ufer domain of Krull dimension two possessing a free $\ZZ$-module basis consisting of the binomial coefficient polynomials $1, X, {X \choose 2}, {X \choose 3}, \ldots$.  Its prime spectrum and Picard group, for example, are also known \cite{cah}.  In the late 1910s, P\'olya and Ostrowski initiated the study of further rings of integer-valued polynomials, namely, $\Int(\mathcal{O}_K)$ for the ring $\mathcal{O}_K$ of integers in a number field $K$.   There is an extensive literature on integer-valued polyomials rings over integral domains, including the two texts \cite{cah} and \cite{nar} witten in the 1990s.  

Roughly a decade ago, the author proved a few elementary universal properties of $\Int(D)$ for any integral domain $D$ \cite{ell0}.  More recently the author proved some less elementary universal properties of $\Int(R)$ for specific classes of rings $R$ \cite{ell}.  Consider the following conditions on a ring $R$.
\begin{enumerate}
\item $\Int(R)$ is free as an $R$-module.
\item  $\Int(R)_\ppp$ is free as an $R_\ppp$-module and equals $\Int(R_\ppp)$ for every maximal ideal $\ppp$ of $R$.
\item $\Int(R)$ has a unique structure of an {\it $R$-plethory} \cite{bor} with unit given by the inclusion $R[X] \longrightarrow \Int(R)$.
\item $\Int(R)$ is the largest $R$-plethory contained in $T(R)[X]$.
\item $\Int(R)$ left-represents a right adjoint for the inclusion from $\mathsf{C}$ to the category $R$-$\mathsf{Alg}$ of $R$-algebras for a unique full, bicomplete, and bireflective subcategory $\mathsf{C}$ of $R$-$\mathsf{Alg}$ (and an $R$-algebra $A$ is in $\mathsf{C}$ if and only if for all $a \in A$ there is a unique $R$-algebra homomorphism $\Int(R) \longrightarrow A$ sending $X$ to $a$).
\end{enumerate}
In general, conditions (3)--(5) are equivalent, and they  hold if either (1) or (2) holds  \cite[Theorems 2.9 and 7.11]{ell}.  Moreover, conditions (4) and (5), when they hold, each provide a universal property for the ring $\Int(R)$.   It is known, for example, that conditions (1) and (2) hold if $R$ is a Dedekind domain or a UFD, and condition (2) holds if $R$ is a Krull domain or more generally a domain of Krull type \cite[Theorem 7.11]{ell}.   Thus, the ring $\Int(R)$ can be characterized by the universal properties (4) and (5) for any ring $R$ satisfying any of the conditions (1)--(3), which includes, for example, any domain $R$ of Krull type.

The author has conjectured that there exist integral domains $D$ such that $\Int(D)$ is not free as a $D$-module, or more generally such that the equivalent conditions (3)--(5) do not hold for $R = D$ \cite{ell, ell2}.   For the two universal properties (4) and (5) of $\Int(R)$ to hold one need not require that $R$ be a domain, since if $R_1, \ldots, R_n$ are rings satisfying condition (4) (or $(5)$) then by  \cite[Corollary 7.15]{ell} the direct product $\prod_{i = 1}^n R_i$ also satisfies conditions (4) and (5), and one has $\Int\left(\prod_{i = 1}^n R_i\right) \cong \prod_{i = 1}^n \Int(R_i)$.  The fact that these universal properties of $\Int(R)$ can be easily shown to hold for some rings $R$ with zerodivisors but not for others provides motivation for the study of integer-valued polynomials over such rings.

Beyond the direct products mentioned above, our first attempt (with students Ryan Hoffman, Cybill Kreil, and Alexander McBroom) was to compute the ring $\Int(R[\varepsilon])$ of integer-valued polynomials on the ring $R[\varepsilon] = R[T]/(T^2)$ of dual numbers over $R$, where $\varepsilon$ is the image of $T$ in $R[T]/(T^2)$.  In Section 2 we compute more generally the ring of integer-valued polynomials over the ring
$$R[\varepsilon_1, \ldots, \varepsilon_n] = R[T_1,\ldots, T_n]/(T_1^2, \ldots, T_n^2)$$
of {\it $n$-hyper-dual numbers over $R$},  where $\varepsilon_i$ for all $i$ is the image of $T_i$ in the given quotient ring.    A special feature of this example is that it provides an independent motivation for the study of integer-valued derivatives, as studied, for example, in \cite[Chapter IX]{cah}.   Another feature is that it implies that conditions (1)--(5) do not hold for $R = \ZZ[\varepsilon]$.

Let $A$ a subset of $T(R)$.  We define
$$\Int(A,R) = \Int^{(0)}(A,R) = \{f \in T(R)[X]: f(A) \subseteq R\}$$
and 
$$\Int^{(1)}(A, R) = \{f \in \Int(A, R): f' \in \Int(A,R)\},$$
where $f'$ denotes the (formal) derivative of $f$.  The former is a subring of $T(R)[X]$ called the ring of {\it integer-valued polynomials on $A$ over $R$}, and the  latter is the subring of $\Int(A,R)$ consisting of all integer-valued polynomials $f$ on $A$ over $R$ whose derivative $f'$ is also integer-valued on $A$.  More generally, we define
\begin{align*}
\Int^{(n)}(A,R) & = \{f \in T(R)[X]: f, f', f'', \ldots, f^{(n)} \in \Int(A,R)\} \\ & = \{f \in \Int(A,R): f' \in \Int^{(n-1)}(A,R) \}
\end{align*}
for all positive integers $n$, and also
$$\Int^{(\infty)}(A,R) = \{f \in T(R)[X]: f, f', f'', \ldots \in \Int(A,R)\} = \bigcap_{n = 0}^\infty \Int^{(n)}(A, R).$$  We write $\Int^{(n)}(R)$ for $\Int^{(n)}(R,R)$ for all $n \in \ZZ_{\geq 0} \cup \{\infty\}$.  The rings $\Int^{(n)}(R)$ are well-studied in the case where $R$ is a domain \cite[Chapter IX]{cah}.

As an $R$-module, the ring $R[\varepsilon_1, \ldots, \varepsilon_n]$ is free of rank $2^n$, with a basis consisting of the elements $\varepsilon_S = \prod_{i \in S} \varepsilon_i$ for all subsets $S$ of the set $\{1,2,\ldots, n\}$; thus, as $R$-modules, one has $$R[\varepsilon_1,\ldots,\varepsilon_n] = \bigoplus_{S \subseteq \{1,2,\ldots,n\}} R \varepsilon_S.$$  The following theorem is proved in Section 2.

\begin{theorem}[with Ryan Hoffman, Cybill Kreil, and Alexander McBroom]\label{hyperdualtheorem}
Let $R$ be a ring, let $n$ be a positive integer, and let $k$ be a nonnegative integer.  One has  
$$\Int^{(k)}(R[\varepsilon_1, \ldots, \varepsilon_n]) = \sum_{i = 0}^n (\varepsilon_1,\ldots, \varepsilon_n)^i \Int^{(n+k-i)}(R)[\varepsilon_1, \ldots, \varepsilon_n].$$
Equivalently, as $R$-modules, one has
$$\Int^{(k)}(R[\varepsilon_1, \ldots, \varepsilon_n]) = \bigoplus_{S \subseteq \{1,2,\ldots,n\}}  \Int^{(n+k-|S|)}(R) \varepsilon_S.$$
\end{theorem}

\begin{corollary}\label{theorem1}
Let $R$ be a ring.  Then $\Int^{(k)}(R,R[\varepsilon]) = \Int^{(k)}(R)[\varepsilon]$ and $\Int^{(k)}(R[\varepsilon]) = \Int^{(k+1)}(R) + \Int^{(k)}(R)\varepsilon$ for all nonnegative integers $k$.   Moreover, one has $\Int^{(\infty)}(R,R[\varepsilon]) =  \Int^{(\infty)}(R[\varepsilon]) = \Int^{(\infty)}(R)[\varepsilon]$.   
\end{corollary}

\begin{corollary}
Let $R$ be a ring and  $n$ a positive integer.  One has
$$\Int^{(\infty)}(R[\varepsilon_1, \ldots, \varepsilon_n]) = \Int^{(\infty)}(R,R[\varepsilon_1, \ldots, \varepsilon_n]) =\Int^{(\infty)}(R)[\varepsilon_1, \ldots, \varepsilon_n].$$
\end{corollary}

Theorem \ref{hyperdualtheorem} provides an explicit realization of the rings $\Int^{(k)}(R)$ for $k \leq n$ in terms of the ring $\Int(R[\varepsilon_1,\ldots,\varepsilon_n])$.  This provides motivation for the study of integer-valued polynomials on rings with zerodivisors, or alternatively for the study of the rings $\Int^{(k)}(R)$, either from the standpoint of the other. Moreover, by \cite[Theorem 2.4 and Proposition 7.28]{ell}, the second corollary above implies that if $\Int^{(\infty)}(R)$ has the structure of an $R$-plethory---which holds if $R$ is a Krull domain, by \cite[Theorem 7.20]{ell}---then $\Int^{(\infty)}(R[\varepsilon_1, \ldots, \varepsilon_n])$ has the structure of an $R[\varepsilon_1, \ldots, \varepsilon_n]$-plethory.  In particular, $\Int^{(\infty)}(\ZZ[\varepsilon]) = \Int^{(\infty)}(\ZZ)[\varepsilon]$ has the structure of a $\ZZ[\varepsilon]$-plethory even though $\Int(\ZZ[\varepsilon])$ does not.

Let $R$ be a ring, let $r_1, \ldots, r_n \in R$, and let $$R[\rho_1, \ldots, \rho_n] = R[T_1, \ldots, T_n]/(T_1(T_1-r_1),\ldots, T_n(T_n-r_n)),$$ where $\rho_i$ is the image of $T_i$ in the given quotient ring for all $i$.  Note that $$R[\rho_1, \ldots, \rho_n] = R[\rho_1][\rho_2]\cdots[\rho_n] = R[\rho_1] \otimes_R \cdots \otimes_R R[\rho_n].$$ Let $r = r_i$ and $\rho = \rho_i$ for a fixed $i$.  The map $Q: R^2 \longrightarrow R$ given by $(x,y) \longmapsto x^2 + rxy$ is a rank two quadratic form over $R$, and the ring $R[\rho]$ is the even Clifford algebra of $Q$.  If $r$ is a unit, then $R[\rho] \cong R \times R$.   If $r = 0$, then $R[\rho] = R[T]/(T^2)$ is the ring of dual numbers over $R$.  If $r = 2$, then $R[\rho] = R[j] = R[U]/(U^2-1)$, where $j = \rho-1$ and $j^2 = 1$, is the ring of {\it split complex numbers over $R$}.   Our main result, Theorem \ref{hypertheorem} of Section 3, generalizes Theorem \ref{hyperdualtheorem} by computing $\Int(R[\rho_1, \ldots, \rho_n])$.

In Section 4, we generalize Corollary \ref{theorem1} by computing $\Int^{(k)}(R(+)M)$ for any $R$-module $M$ satisfying a certain regularity condition, where $R(+)M$ denotes the (Nagata) idealization of $M$ over $R$.    (Note that $R[\varepsilon] \cong R(+)R$.)   This provides motivation for a generalization of integer-valued polynomials to modules.  The final section, Section 5, connects such a generalization to the notion of polynomial torsion in modules.

\section{Hyper-dual-integer-valued  polynomials}

Let $R$ be a ring.  The ring $R[\varepsilon] = R[T]/(T^2)$, where $\varepsilon$ denotes the image of $T$ in the given quotient ring, is called the {\it ring of dual numbers over $R$}.  Its elements are expressions of the form $x+y\varepsilon$ with $x,y \in R$, where $\varepsilon^2 = 0$.  The ring $R$ is a subring of $R[\varepsilon]$ under the identification $x = x + 0\varepsilon$ for all $x \in R$.   For all $z = x+y\varepsilon \in R[\varepsilon]$, one defines the {\it conjugate} of $z$ to be $\overline{z} = x-y\varepsilon$  and the {\it modulus} of $z$ to be  $|z| = x$.  Then one has $z \overline{z} = |z|^2$, as well as $|zw| = |z||w|$ and $\overline{zw} = \overline{z} \, \overline{w}$, for all $z,w \in R[\varepsilon]$, in analogy with the complex numbers.

Let $n$ a positive integer.  Recall from the introduction that we let $$R[\varepsilon_1, \ldots, \varepsilon_n] = R[T_1,\ldots, T_n]/(T_1^2, \ldots, T_n^2),$$ where $\varepsilon_i$ for all $i$ is the image of $T_i$ in the given quotient ring.    Note that $R[\varepsilon_1, \ldots, \varepsilon_n] = R[\varepsilon_1,\ldots,\varepsilon_{n-1}][\varepsilon]$, where $\varepsilon_n = \varepsilon$. 

 In this section we provide two proofs of Theorem \ref{hyperdualtheorem}, which computes $\Int^{(k)}(R[\varepsilon_1, \ldots, \varepsilon_n])$ for all positive integers $n$ and all nonnegative integers $k$.   Our first proof is by induction on $n$.   We first prove the base case $n = 1$ by computing $\Int(R[\varepsilon])$.  To this end we provide the following three elementary lemmas.

\begin{lemma}
Let $R$ be a ring and $z \in R[\varepsilon]$.  Then $z$ is a non-zerodivisor (resp., unit) of $R[\varepsilon]$ if and only if $|z|$ is a non-zerodivisor (resp., unit) of $R$.
\end{lemma}


\begin{lemma}\label{dualtotalquotientring}
Let $R$ be a ring.  Then the total quotient ring $T(R[\varepsilon])$ of $R[\varepsilon]$ is given by $T(R)[\varepsilon]$, where $T(R)$ is the total quotient ring of $R$.
\end{lemma}

\begin{proof}
If $z = x+y\varepsilon$ is a non-zerodivisor of $R[\varepsilon]$, then $x$ is a non-zerodivisor of $R$, whence $x$ is a unit of $T(R)$ and therefore $z = x+y\varepsilon$ is a unit of $T(R)[\varepsilon]$.  Therefore every non-zerodivisor of $R[\varepsilon]$ is a unit of $T(R)[\varepsilon]$.  Thus $T(R[\varepsilon])$ is a subring of $T(R)[\varepsilon]$.  Let $w \in T(R)[\varepsilon]$, so $w = xu^{-1} +y v^{-1} \varepsilon$, where $x, y, u, v \in R$ and $u$ and $v$ are non-zerodivisors of $R$.  Then $w = (vx+uy\varepsilon)(uv)^{-1}$, where $vx + uy \in R[\varepsilon]$, and where $uv$ is a non-zerodivisor of $R$, hence a non-zerodivisor of $R[\varepsilon]$.  Therefore $w \in T(R[\varepsilon])$.  The lemma follows.
\end{proof}

\begin{lemma}\label{dualpolynomialfunctions}
Let $R$ be a ring.  Every $F \in R[\varepsilon][X]$ can be written uniquely in the form $F = f + g\varepsilon$, where $f , g \in R[X]$.  Moreover, one has
$$F(x + y \varepsilon) = f(x) + (f'(x)y+g(x)) \varepsilon$$
for all $x, y \in R$.
\end{lemma}

\begin{proof}
We may write $F  = \sum_{k = 0}^n (f_k + g_k \varepsilon) X^k = f+ g \varepsilon$, where $f = \sum_{k = 0}^n f_k X^k \in R[X]$ and $g = \sum_{k = 0}^n g_k X^k \in R[X]$, and clearly $f$ and $g$ are unique.  Let $k$ be a positive integer.  By the binomial theorem one has $(x+y\varepsilon)^k = x^k + kx^{k-1} y \varepsilon.$
In other words, if $h = X^k$, then one has $h(x+y \varepsilon) = h(x)+h'(x)y\varepsilon$.  Therefore, extending this by linearity we see that $f(a+b \varepsilon) = f(a)+f'(a)b\varepsilon$ and $g(x+y \varepsilon) = g(x)+g'(x)y\varepsilon$.  Therefore we have
$F(x + y \varepsilon) = f(x+y\varepsilon) + (g(x+y\varepsilon)) \varepsilon  = f(x) + (f'(x)y+g(x)) \varepsilon,$  as claimed.
\end{proof}

\begin{proposition}[with Ryan Hoffman, Cybill Kreil, and Alexander McBroom]\label{dualintegervaluedpolynomials}
One has
$$\Int(R[\varepsilon]) = \Int^{(1)}(R) + \Int(R)\varepsilon$$
for any ring $R$.
\end{proposition}

\begin{proof}
By Lemma \ref{dualtotalquotientring}, the total quotient ring of $R[\varepsilon]$ is $K[\varepsilon]$, where $K$ is the total quotient ring of $R$.  Let $F \in K[\varepsilon][X]$.  By Lemma \ref{dualpolynomialfunctions}, we may write $F = f + g \varepsilon$ with $f, g \in K[X]$, and one has $F(x+y \varepsilon) = f(x) + (f'(x)y+g(x)) \varepsilon$
for all $x,y \in K$.  Therefore, $F$ lies in $\Int(R[\varepsilon])$ if and only if both $f(x)$ and $f'(x)y + g(x)$ lie in $R$ for all $x,y \in R$, if and only if $f(x)$, $f'(x)$, and $g(x)$ lie in $R$ for all $x \in R$, if and only if $f \in \Int^{(1)}(R)$ and $g \in \Int(R)$.  Thus, $F \in K[\varepsilon][X] = T(R[\varepsilon])[X]$ lies in $\Int(R[\varepsilon])$ if and only if $F$ is of the form $f+g \varepsilon$ with $f \in \Int^{(1)}(R)$ and $g \in \Int(R)$.  The proposition  follows.
\end{proof}

As an $R$-module, the ring $R[\varepsilon_1, \ldots, \varepsilon_n]$ is free of rank $2^n$, with a basis consisting of the elements $\varepsilon_S = \prod_{i \in S} \varepsilon_i$ for all subsets $S$ of the set $\{1,2,\ldots, n\}$. 

\begin{proof}[{Proof of Theorem \ref{hyperdualtheorem}}]
The statement for arbitrary $k$ follows readily from the statement for $k = 0$, so it suffices to prove the statement for $k = 0$.  The proof is by induction on $n$.  The statement for $n = 1$ is Proposition \ref{dualintegervaluedpolynomials}.  Suppose the statement holds for $n-1$ for some $n >1$.  Then we have
\begin{align*}
\Int(R[\varepsilon_1, \ldots, \varepsilon_n] ) & = \Int(R[\varepsilon_1,\ldots,\varepsilon_{n-1}][\varepsilon_n]) \\
&  = \Int^{(1)}(R[\varepsilon_1,\ldots,\varepsilon_{n-1}]) + \Int(R[\varepsilon_1,\ldots,\varepsilon_{n-1}]) \varepsilon_n \\
& = \bigoplus_{S \subseteq \{1,2,\ldots,n-1\}}  \Int^{(n-1+1-|S|)}(R) \varepsilon_S + \left(\bigoplus_{S \subseteq \{1,2,\ldots,n-1\}}  \Int^{(n-1-|S|)}(R) \varepsilon_S \right) \varepsilon_n  \\
& = \bigoplus_{n \notin S \subseteq \{1,2,\ldots,n\}}  \Int^{(n-|S|)}(R) \varepsilon_S + \bigoplus_{n \in S \subseteq \{1,2,\ldots,n\}}  \Int^{(n-|S|)}(R) \varepsilon_S  \\ 
& = \bigoplus_{S \subseteq \{1,2,\ldots,n\}}  \Int^{(n-|S|)}(R) \varepsilon_S.
\end{align*}
This completes the proof.
\end{proof}

We now provide a more direct proof of a stronger version of Theorem \ref{hyperdualtheorem} that avoids induction on $n$.

\begin{lemma}\label{hyperdualtotalquotientring}
Let $R$ be a ring and $n$ a positive integer.  Then the total quotient ring $T(R[\varepsilon_1, \ldots, \varepsilon_n])$ of $R[\varepsilon_1, \ldots, \varepsilon_n]$ is given by $T(R)[\varepsilon_1, \ldots, \varepsilon_n]$.
\end{lemma}

\begin{lemma}[with Ryan Hoffman, Cybill Kreil, and Alexander McBroom]\label{hyperdualpolynomialfunctions}
Let $R$ be a ring and $n$ a positive integer.  One has $$R[\varepsilon_1, \ldots, \varepsilon_n][X] = R[X][\varepsilon_1, \ldots, \varepsilon_n],$$ so every $F \in R[\varepsilon_1, \ldots, \varepsilon_n][X]$ can be written uniquely in the form $F = \sum_{S \subseteq \{1,2,\ldots,n\}} f_S \varepsilon_S$, where $f_S \in R[X]$ for all $S \subseteq  \{1,2,\ldots,n\}$.  Moreover, one has
\begin{align} F(x + x_1 \varepsilon_1 + x_2 \varepsilon_2 + \cdots + x_n \varepsilon_n) & = \sum_{S \subseteq \{1,2,\ldots,n\}}  F^{(|S|)}(x) \prod_{i \in S} x_i \varepsilon_i  \label{eqc} \\ & = \sum_{S \subseteq \{1,2,\ldots,n\}}  \left( \sum_{T \subseteq S} f_{S-T}^{(|T|)}(x) \prod_{i \in T} x_i \right) \varepsilon_S  \label{eqd}
\end{align}
for all $x, x_1, x_2, \ldots, x_n \in R$.  More generally, for all $\xx = \sum_{S \subseteq \{1,2,\ldots,n\}} x_S \varepsilon_S$, one has
\begin{align} F(\xx)  &= \sum_{S \subseteq \{1,2,\ldots,n\}} \left(\sum_{{\textup{partition }} \atop {S_1, \ldots, S_k {\textup{ of } S}}} F^{(k)}(x_\emptyset)\prod_{i = 1}^k x_{S_i} \right) \varepsilon_S \label{eqa}  \\
 & = \sum_{S \subseteq \{1,2,\ldots,n\}} \left(\sum_{T \subseteq S} \sum_{{\textup{partition }} \atop {T_1, \ldots, T_k {\textup{ of } T}}} f_{S-T}^{(k)}(x_\emptyset)\prod_{i = 1}^k x_{T_i} \right) \varepsilon_S. \label{eqb}
\end{align}
\end{lemma}

\begin{proof}
Let $x = x_\emptyset$, so $\xx-x = \sum_{T \neq \emptyset} x_T \varepsilon_T$.  We may assume without loss of generality that $R$ is a polynomial ring over $\ZZ$ and is therefore $\ZZ$-torsion-free.  Then, by Taylor's theorem and the multinomial theorem, we have the following.
\begin{align*} \displaystyle
F(\xx) & = \sum_{k = 0}^\infty F^{(k)}(x) \frac{1}{k!}(\xx-x)^k \\
  & = \sum_{k = 0}^\infty F^{(k)}(x) \frac{1}{k!} \sum_{\sum_{T \neq \emptyset} k_T = k} \frac{k!}{\prod_{T \neq \emptyset} k_T!} \prod_{T \neq \emptyset} (x_T \varepsilon_T)^{k_T} \\
 & = \sum_{k = 0}^\infty F^{(k)}(x) \frac{1}{k!} \sum_{\sum_{T \neq \emptyset} k_T = k \atop {\forall T \, k_T \in \{0,1\}}} k! \prod_{T \neq \emptyset} (x_T \varepsilon_T)^{k_T} \\
 & = \sum_{k = 0}^\infty F^{(k)}(x) \sum_{\sum_{T \neq \emptyset} k_T = k \atop {
\forall T \, k_T \in \{0,1\}}} \prod_{T \neq \emptyset} (x_T \varepsilon_T)^{k_T} \\
 & = \sum_{k = 0}^n F^{(k)}(x)\sum_{\emptyset \neq S \subseteq \{1,\ldots,n\}} \sum_{\textup{partition } \atop{S_1, \ldots, S_k} \textup{ of } S}  \left( \prod_{i = 1}^k x_{S_i} \right) \varepsilon_S \\
 & = \sum_{ S \subseteq \{1,\ldots,n\}} \left( \sum_{\textup{partition } \atop{S_1, \ldots, S_k} \textup{ of } S}   F^{(k)}(x) \prod_{i = 1}^k x_{S_i} \right) \varepsilon_S. \\
\end{align*}
This proves Eq.\ \ref{eqa}, from which the rest of the lemma readily follows.
\end{proof}

\begin{remark}
Let $\HH^n_K = K[\varepsilon_1, \ldots, \varepsilon_n]$, where $K$ is $\RR$ or $\CC$.  The lemma allows one to extend any partial function $F: K \longrightarrow K$ to a partial function $F: \HH^n_K \longrightarrow \HH^n_K$ defined on the set $\displaystyle\left\{z \in \HH^n_K: F^{(n)}(|z|) \textup{ exists}\right\}$, where $|z|$ denotes the coefficient $x_\emptyset$ of $z = \sum_S x_S \varepsilon_S$.
\end{remark}

\begin{example}
For $n = 2$ one has
$$F(x+x_1\varepsilon_1 + x_2 \varepsilon_2 + x_{12}\varepsilon_{1} \varepsilon_2) = F(x) + F'(x)x_1 \varepsilon_1 + F'(x)x_2 \varepsilon_2 + (F''(x)x_1x_2+ F'(x)x_{12}) \varepsilon_1 \varepsilon_2.$$  For $n = 3$ one has that the expression
$$F(x+x_1\varepsilon_1 + x_2\varepsilon_2+ x_3\varepsilon_3 +   x_{12}\varepsilon_1 \varepsilon_2 +  x_{13}\varepsilon_1 \varepsilon_3 +  x_{23}\varepsilon_2 \varepsilon_3 +   x_{123}\varepsilon_1 \varepsilon_2 \varepsilon_3)$$
is equal to the sum of $F(x)$, the linear terms $F'(x)x_1 \varepsilon_1 +  F'(x)x_2 \varepsilon_2 +  F'(x)x_3 \varepsilon_3,$
the quadratic terms
$$(F''(x)x_1x_2 + F'(x)x_{12})\varepsilon_1 \varepsilon_2 +  (F''(x)x_1x_3 + F'(x)x_{13})\varepsilon_1 \varepsilon_3 + (F''(x)x_2x_3 + F'(x)x_{23})\varepsilon_2 \varepsilon_3,$$
and the cubic term
$$(F'''(x)x_1 x_2 x_3 + F''(x) x_{12}x_3 + F''(x) x_{13}x_2 + F''(x)x_{23}x_1 + F'(x)x_{123}) \varepsilon_1 \varepsilon_2 \varepsilon_3.$$
\end{example}

\begin{remark}
There are a total of $B_{n+1} = \sum_{k = 0}^n {n \choose k} B_k$ terms in the double sum of Eq.\ \ref{eqa}, while there are a total of $B_{n+2}-B_{n+1} = \sum_{k = 0}^n \sum_{l = 0}^k {n \choose k} {k \choose l} B_l$ terms in the triple sum of Eq.\ \ref{eqb},  where $B_k$ denotes the $k$th Bell number, equal to the number of partitions of a $k$-element set.  Also, there are a total of $2^n$ terms in the sum of Eq.\ \ref{eqc} and a total of $3^n$ terms in the double sum of Eq.\ \ref{eqd}.
\end{remark}

\begin{theorem}[with Ryan Hoffman, Cybill Kreil, and Alexander McBroom]\label{hyperdualtheorem2}
Let $R$ be a ring.  For all positive integers $n$ and all nonnegative integers  $k$ one has
 $$\Int^{(k)}(R,R[\varepsilon_1, \ldots, \varepsilon_n]) =  \bigoplus_{S \subseteq \{1,2,\ldots,n\}} \Int^{(k)}(R) \varepsilon_S,$$ 
$$\Int^{(k)}(R[\varepsilon_1, \ldots, \varepsilon_n]) = \bigoplus_{S \subseteq \{1,2,\ldots,n\}}  \Int^{(n+k-|S|)}(R) \varepsilon_S,$$ 
$$\Int^{(k)}(R[\varepsilon_1, \ldots, \varepsilon_n])  = \Int^{(k)}\left(R + \sum_{i = 1}^n \{0,1\} \varepsilon_i, \, R[\varepsilon_1, \ldots, \varepsilon_n]\right).$$
\end{theorem} 

\begin{proof}
Let $\mathcal{A} = \bigoplus_{S \subseteq \{1,2,\ldots,n\}}  \Int^{(n-|S|)}(R) \varepsilon_S$.
The theorem can be reduced to showing that $$\mathcal{A} = \Int(R[\varepsilon_1, \ldots, \varepsilon_n]) = \Int\left(R + \sum_{i = 1}^n \{0,1\} \varepsilon_i, \, R[\varepsilon_1, \ldots, \varepsilon_n]\right).$$   Each of these three rings is a subring of  the ring $K[\varepsilon_1, \ldots, \varepsilon_n][X]$, where $K = T(R)$.  Let $F = \sum_S f_S \varepsilon_S \in K[\varepsilon_1, \ldots, \varepsilon_n][X]$.    Suppose that $F \in \mathcal{A}$, so that $f_S^{(k)} \in \Int(R)$ for all $S \subseteq \{1,2,\ldots,n\}$ and all $k$ with $|S| + k \leq n$.  For all $T \subseteq S \subseteq \{1,2,\ldots,n\}$ such that $T$ has a partition into $k$ subsets, one has $|T| \geq k$ and therefore $|S-T| + k = |S|-|T|+k \leq n -(|T|-k) \leq n$.  Therefore  $f_{S-T}^{(k)} \in \Int(R)$ for all such $S,T,k$.  It follows from Lemma \ref{hyperdualpolynomialfunctions} that $$ F(\xx) = \sum_{S \subseteq \{1,2,\ldots,n\}} \left(\sum_{T \subseteq S} \sum_{{\textup{partition }} \atop {T_1, \ldots, T_k {\textup{ of } T}}} f_{S-T}^{(k)}(x_\emptyset)\prod_{i = 1}^k x_{T_i} \right) \varepsilon_S  \in R[\varepsilon_1, \ldots, \varepsilon_n]$$
for all $\xx = \sum_S x_S \varepsilon_S \in R[\varepsilon_1, \ldots, \varepsilon_n]$, and therefore $F \in \Int(R[\varepsilon_1, \ldots, \varepsilon_n])$. Therefore one has
$$\mathcal{A} \subseteq \Int(R[\varepsilon_1, \ldots, \varepsilon_n]) \subseteq \Int\left(R + \sum_{i = 1}^n \{0,1\} \varepsilon_i, \, R[\varepsilon_1, \ldots, \varepsilon_n]\right).$$  It remains only to show, then, that if $F \in \Int\left(R + \sum_{i = 1}^n \{0,1\} \varepsilon_i, \, R[\varepsilon_1, \ldots, \varepsilon_n]\right)$, then $F \in \mathcal{A}$.  By Lemma \ref{hyperdualpolynomialfunctions} one has
$$F(x + x_1 \varepsilon_1 + x_2 \varepsilon_2 + \cdots + x_n \varepsilon_n) = \sum_{S \subseteq \{1,2,\ldots,n\}}  \left( \sum_{T \subseteq S} f_{S-T}^{(|S|)}(x) \prod_{i \in T} x_i \right) \varepsilon_S \in R[\varepsilon_1, \ldots, \varepsilon_n],$$
and therefore  $$\sum_{T \subseteq S} f_{S-T}^{(|T|)}(x) \prod_{i \in T} x_i \in R \mbox{ for all } S \subseteq \{1,2,\ldots,n\},$$
for all $x \in R$ and all $x_1, \ldots, x_n \in \{0,1\}$.  We claim that each of the terms $f_{S-T}^{(|T|)}(x)$ appearing in the above sum (for all $T \subseteq S \subseteq \{1,2,\ldots,n\}$) must lie in $R$ and therefore $F \in \mathcal{A}$.  The proof for fixed $S$ and $x$ is by induction on $|T|$.  If $|T| = 0$, then $T = \emptyset$.  Letting $x_i = 0$ for all $i$, we see that 
$$f_S^{(0)}(x) = \sum_{V\subseteq S} f_{S-V}^{(|V|)}(x) \prod_{i \in V} x_i \in R,$$
which proves the base case.  Suppose that  $f_{S-T}^{(|S|)}(x) \in R$ for subsets $T$ of $S$ of cardinality at most $k-1$.  Let $T$ be a subset of $S$ of cardinality $k$.  Let 
\begin{eqnarray*}
  x_i  = \left\{
     \begin{array}{ll}
       1  & \mbox{if } i \in T\\
       0 & \mbox{otherwise}.
     \end{array}
   \right.
\end{eqnarray*}
Then $$f_{S-T}^{(|T|)}(x) + \sum_{V\subsetneq T} f_{S-V}^{(|V|)}(x)  =  \sum_{V\subseteq T} f_{S-V}^{(|V|)}(x) =  \sum_{V\subseteq S} f_{S-V}^{(|V|)}(x) \prod_{i \in V} x_i \in R,$$
while by the inductive hypothesis one has $\sum_{V\subsetneq T} f_{S-V}^{(|V|)}(x)  \in R$.  Therefore $f_{S-T}^{(|T|)}(x)  \in R$.
\end{proof}

\begin{remark}
A subset $A$ of a ring $R$ is said to be {\it  polynomially dense in $R$} if $\Int(A,R) = \Int(R)$.  Theorem \ref{hyperdualtheorem2} is equivalent to the conjunction of Theorem \ref{hyperdualtheorem} and the fact that the set $R+\sum_{i=1}^n \{0,1\} \varepsilon_i$ is polynomially dense  in $R[\varepsilon_1,\ldots, \varepsilon_n]$.
\end{remark}

\section{Integer-valued polynomials on generalizations of the hyper-dual numbers}

Let $R$ be a ring.  Consider the ring $R[j] = R[T]/(T^2-1)$, where $j$ is the image of $T$ in the given quotient ring.  The ring $\RR[j]$ is called the ring of {\it split complex numbers}, or {\it hyperbolic numbers}.  We  call $R[j]$ the ring of {\it split complex numbers over $R$}.   Note that $R[j] = R[1+j] = R[U]/(U(U-2))$, where $j$ is the image of $U-1$ in the given quotient ring. 

The dual numbers and split complex numbers over $R$ generalize as follows.   Let $r$ be a fixed element of $R$, and let $R[\rho] = R[T]/(T(T-r))$, where $\rho$ is the image of $T$ in the given quotient ring.    If $r = 0$, then we obtain the ring of dual numbers $R[\rho] = R[\varepsilon]$, where $\rho = \varepsilon$.   If $r = 2$, then we obtain the ring of split complex numbers $R[\rho] = R[j]$, where $\rho = 1+j$.   If $r$ is a unit, then $R[\rho] \cong R \times R$.  

For all $z = x+y\rho \in R[\rho]$ we define $\overline{z} = x+y(r-\rho)$ and $||z|| = z\overline{z} = x(x+ry) = x^2+rxy$.  Note that $\overline{\overline{z}} = z$ and $\overline{zw} = \overline{z}\ \overline{w}$, and therefore $||zw|| = ||z|| \, ||w||$, for all $z,w \in R[\rho]$.
The map $Q: R^2 \longrightarrow R$ given by $(x,y) \longmapsto x^2 + rxy$ is a rank two quadratic form over $R$, and the ring $R[\rho]$ is the even Clifford algebra of $Q$.

\begin{lemma}
Let $R$ be a ring and $z \in R[\rho]$.  Then $z = x+y\rho$ is a non-zerodivisor (resp., unit) of $R[\rho]$ if and only if $||z||$ is a non-zerodivisor (resp., unit) of $R$, if and only if $x$ and $x+ry$ are non-zerodivisors (resp., units) of $R$.
\end{lemma}

\begin{lemma}\label{dualtotalquotientring}
Let $R$ be a ring.  Then the total quotient ring $T(R[\rho])$ of $R[\rho]$ is given by $T(R)[\rho]$, where $T(R)$ is the total quotient ring of $R$.
\end{lemma}

For all $f \in R[X]$, we may write $f(X+Y)-f(X) = Yg(X,Y)$ for a unique $g \in R[X,Y]$.  We write $\Delta_Y f(X) = g(X,Y) \in R[X,Y]$, so that $\Delta_Y f(X)$ denotes $\frac{f(X+Y)-f(X)}{Y}$ but is a polynomial in $X$ and $Y$.   We may then define $\Delta_y f(X) = g(X,y) \in R[X]$, whence $\Delta_y$ is an $R$-linear operator on $R[X]$,  for all $y \in R$.  One has $\Delta_Y f(X) = f'(X) + YG(X,Y)$ for some $G \in R[X,Y]$, and therefore $\Delta_0 f(X) = f'(X)$.  One has the following generalization of the product and chain rules for derivatives:
$$\Delta_Y(f \cdot g)(X) = \Delta_Y f(X) \cdot g(X+Y)+ f(X) \cdot \Delta_Y g(X)$$
and
$$\Delta_Y(f \circ g)(X) = (\Delta_{g(X+Y))-g(X)} f)(g(X)) \Delta_Y g(X)$$
for all $f, g \in R[X]$.

\begin{lemma}\label{hyperbolicpolynomialfunctions}
Let $R$ be a ring, let $r \in R$, and let $R[\rho] = R[T]/(T(T-r))$, where $\rho$ is the image of $T$ in the given quotient ring.  For all $F \in R[\rho][X]$ and all $x,y \in R$, one has
$$F(x + y \rho) = F(x) + \Delta_{yr}F(x) \cdot y\rho = F(x+yr)+y\Delta_{yr}F(x) \cdot (\rho-r).$$
\end{lemma}

\begin{proof}
As with Lemma \ref{dualpolynomialfunctions}, one can verify the lemma for $F = X^n$ and then extend by linearity.
\end{proof}

For any $r \in R$, we let 
\begin{align*} \Int(R;r) & = \{f \in \Int(R): \forall y \in R \ (y \Delta_{yr}f(X) \in \Int(R))\}.
\end{align*}
Note, for example, that $$\Int(R; 0) = \Int^{(1)}(R)$$
and $$\Int(R;1) = \Int(R; r) = \Int(R)$$
if $r$ is a unit of $R$.

\begin{proposition}\label{hyperbolicint}
Let $R$ be a ring, let $r \in R$, and let $R[\rho] = R[T]/(T(T-r))$, where $\rho$ is the image of $T$ in the given quotient ring.  For all $f, g \in T(R)[X]$, one has $f + g \rho \in \Int(R[\rho])$ if and only if $f \in \Int(R; r)$ and $g \in \Int(R)$.
Equivalently, one has $$\Int(R[\rho]) = \Int(R; r) +\Int(R) \rho.$$
\end{proposition}

\begin{proof}
Let $F = f+g\rho$. By Lemma \ref{hyperbolicpolynomialfunctions} one has
\begin{align*}
F(x+y\rho) & = f(x) + y \Delta_{yr}f(x) \cdot \rho + (g(x) + y \Delta_{yr}g(x) \cdot \rho)\rho\\
& = f(x) + (y \Delta_{yr}f(x) + g(x) + yr\Delta_{yr}g(x)) \rho\\
& = f(x) + (y \Delta_{yr}f(x) +g(x+yr))\rho.
\end{align*}
If $F \in \Int(R[\rho])$, then clearly $f,g \in \Int(R)$ (e.g., consider the case $y = 0$), so, subtracting $g(x+yr) \rho \in R[\rho]$ from the equation above we see that $f(x) +  y \Delta_{yr}f(x) \cdot\rho \in R[\rho]$ for all $x, y \in R$ and therefore $f \in \Int(R; r)$.   Conversely, we see immediately from the same equation that if $f \in \Int(R; r)$ and $g \in \Int(R)$ then $F = f+ g\rho \in \Int(R[\rho])$.
\end{proof}

\begin{corollary}
Let $R$ be a ring.   One has 
$$\Int(R[\varepsilon]) = \Int^{(1)}(R) + \Int(R)\varepsilon$$ 
and
$$\Int(R[j]) = \Int(R;2) +\Int(R) (1+j).$$
\end{corollary}

Now, let 
$$\Int^{(1)}(R; r) = \{f \in T(R)[X]: f,f' \in \Int(R;r)\}$$
for any $r \in R$.

\begin{proposition}\label{mixedprop}
Let $R$ be a ring, let $r \in R$, and let $R[\rho, \varepsilon] = R[T,U]/(T(T-r),U^2)$, where $\rho$ is the image of $T$ and $\varepsilon$ is the image of $U$.  One has
$$\Int(R[\rho,\varepsilon]) = \Int^{(1)}(R;r) +  \Int^{(1)}(R) \rho + \Int(R; r) \varepsilon  + \Int(R) \rho \varepsilon.$$ Moreover, one has
$$\Int^{(1)}(R;r) = T(R)[X] \cap \Int(R[\rho,\varepsilon]),$$
$$\Int(R[\varepsilon]; r) = T(R)[\varepsilon][X] \cap \Int(R[\rho,\varepsilon]) =  \Int^{(1)}(R;r) + \Int(R; r) \varepsilon,$$
$$\Int^{(1)}(R[\rho])  = T(R)[\rho][X] \cap \Int(R[\rho,\varepsilon]) = \Int^{(1)}(R;r) +   \Int^{(1)}(R) \rho.$$
\end{proposition}

\begin{proof}
Note that $R[\rho, \varepsilon] = R[\rho][\varepsilon]= R[\varepsilon][\rho]$.  Therefore one has
\begin{align*} 
\Int(R[\varepsilon,\rho]) & = \Int(R[\varepsilon]; r) +  \Int(R[\varepsilon]) \rho\\
& = \Int(R[\varepsilon]; r)  +  \Int^{(1)}(R) \rho + \Int(R) \rho \varepsilon
\end{align*}
and
\begin{align*} 
\Int(R[\rho, \varepsilon]) & = \Int^{(1)}(R[\rho]) +  \Int(R[\rho]) \varepsilon\\
& = \Int^{(1)}(R[\rho]) + \Int(R; r) \varepsilon  + \Int(R) \rho \varepsilon.
\end{align*}
One checks that the elements of  $\Int^{(1)}(R; r)$ are precisely the polynomials in $T(R)[X]$ that lie in $\Int(R[\rho,\varepsilon])$.  The proposition follows by combining these facts.
\end{proof}

To extend Proposition \ref{hyperbolicint} to compute $\Int(R[\rho_1, \ldots, \rho_n])$, we need the following definitions.  We let
$$\Int(R; \ ) = \Int(R),$$
and for any elements $s_1, \ldots, s_{m+1}$ of $R$ we let $$\Int(R; s_1, \ldots, s_{m+1})$$ equal the set of all $f \in \Int(R; s_1, \ldots, s_{m})$ such that $y \Delta_{y s_{m+1} \prod_{i \in \{1,\ldots,m\}-\{i_1, \ldots, i_k\}} s_i} f(X) \in \Int(R; s_{i_1},\ldots, s_{i_k})$ for all $y \in R$ and all $1 \leq i_1 < \cdots < i_k \leq m$, which by induction on $m$ is an $R[X]$-subalgebra of $\Int(R)$.  For example, $\Int(R; r)$ is as defined earlier, and
$$\Int(R; r,s) = \{f \in \Int(R, r): y \Delta_{ys}f(X) \in \Int(R;r) \mbox{ and } y\Delta_{yrs}f(X) \in \Int(R) \mbox{ for all } y \in R \}.$$
Since $\Delta_z \Delta_y = \Delta_y \Delta_z$ as operators on $T(R)[X]$ for all $y,z \in R$, one has
$$\Int(R; r_1, \ldots, r_m) = \Int(R; r_{\sigma(1)},\ldots, r_{\sigma(m)})$$ for any permutation $\sigma$ of $\{1,\ldots,  m\}$, so we may define $$\Int(R; S) = \Int(R; s_1, \ldots, s_m)$$ for any finite multisubset $S = \{s_1, \ldots, s_m\}$ of $R$.   Thus we have
$$\Int(R; \emptyset) = \Int(R),$$
and for any finite multisubset $S$ of $R$ and any $s \in R$ we have
$$\Int(R; S \cup\{s\}) = \{f \in \Int(R; S): y \Delta_{ys \prod_{t \in S-T} t} f(X) \in \Int(R; T)\mbox{ for all }T \subseteq S \mbox{ and } y \in R\}.$$

The following proposition is a straightforward generalization of  Proposition \ref{hyperbolicint}.
\begin{proposition}\label{mixedprop2}
Let $R$ be a ring, let $r \in R$, let $R[\rho] = R[T]/(T(T-r))$, where $\rho$ is the image of $T$ in the given quotient ring, and let $S$ be a finite multisubset of $R$.  Then
$$\Int(R[\rho]; S) = \Int(R; S \cup \{r\}) + \Int(R; S) \rho.$$
\end{proposition}

\begin{proof}
The proof is by strong induction on $|S|$.  The base case $|S| = 0$ is  Proposition \ref{hyperbolicint}.  Suppose  the proposition holds for $|S| \leq m$, and let $s \in R$.  Then
$$\Int(R[\rho]; S \cup\{s\}) = \{F \in \Int(R[\rho]; S): \alpha \Delta_{\alpha s \prod_{t \in S-T} t} F(X) \in \Int(R[\rho]; T)\mbox{ for all }T \subseteq S, \alpha \in R[\rho]\}.$$
Let $F = f+g\rho$ and $\alpha = y + z\rho$, where $f, g \in T(R)[X]$ and $y,z \in R$.   Let $T \subseteq S$ and $u = u_T=  s \prod_{t \in S-T} t$.  
Then
\begin{align*}
\alpha \Delta_{\alpha u} F(X)  & = y \Delta_{yu} F(X) + z\Delta_{zru} F(X) \rho \\
  & = y \Delta_{yu} f(X) + (y \Delta_{yu} g(X) + z\Delta_{zru} f(X) + zr\Delta_{zru} g(X))\rho.
\end{align*}
Now, one has (1) $F \in \Int(R[\rho]; S \cup \{s\})$ if and only if one has (2a) $F \in \Int(R[\rho]; S)$ and (2b) $\alpha \Delta_{\alpha u} F(X) \in \Int(R[\rho]; T)$ for all $\alpha, T$.  By the inductive hypothesis (2a) holds if and only if $f \in \Int(R; S \cup \{r\})$ and $g \in \Int(R; S)$.   Moreover, (2b) holds if and only if $ y \Delta_{yu} f(X) \in \Int(R; T \cup \{r\})$  and $y \Delta_{yu} g(X), z\Delta_{zru} f(X) \in \Int(R; T)$ (hence also $zr\Delta_{zru} g(X) \in \Int(R; T)$ for $y = zr$) for all $y,z, T$.  It follows readily that (1) holds if and only if $f \in \Int(R; S \cup \{s\} \cup \{r\})$ and $g \in \Int(R; S \cup\{s\})$.  This completes the proof.
\end{proof}

The following theorem, which is the main result of this paper, generalizes the previous proposition and Theorem \ref{hyperdualtheorem}.

\begin{theorem}\label{hypertheorem}
Let $R$ be a ring, let $r_1, \ldots, r_n \in R$, let $$R[\rho_1, \ldots, \rho_n] = R[T_1, \ldots, T_n]/(T_1(T_1-r_1),\ldots, T_n(T_n-r_n)),$$ where $\rho_i$ is the image of $T_i$ in the given quotient ring for all $i$, and let $S$ be a finite multisubset of $R$.  Let $\rho_T = \prod_{i \in T} \rho_i$, and let $r_T$ denote the multisubset $\{r_{i_1}, \ldots, r_{i_k}\}$ of $R$, for all $T = \{i_1, \ldots, i_k\} \subseteq \nn = \{1,2,\ldots,n\}$.  One has  
$$\Int(R[\rho_1, \ldots, \rho_n]; S) = \bigoplus_{T \subseteq \nn}  \Int(R; S \cup r_{\nn-T}) \rho_T.$$  In particular, one has
$$\Int(R[\rho_1, \ldots, \rho_n]) = \bigoplus_{T \subseteq \nn}  \Int(R; r_{\nn-T}) \rho_T.$$
\end{theorem}

\begin{proof}
The proof is by induction on $n$.  The statement for $n = 1$ is Proposition \ref{mixedprop2}.  Suppose the statement holds for $n-1$.  Then we have
\begin{align*}
\Int(R[\rho_1, \ldots, \rho_n]; S ) & = \Int(R[\rho_1,\ldots,\rho_{n-1}][\rho_n]; S) \\
&  = \Int(R[\rho_1,\ldots,\rho_{n-1}]; S \cup \{r_n\}) + \Int(R[\rho_1,\ldots,\rho_{n-1}]; S) \rho_n \\
& =\bigoplus_{T \subseteq {\bf n-1}}  \Int(R; S  \cup \{r_n\} \cup r_{{\bf n-1} -T}) \rho_T   + \left(\bigoplus_{T \subseteq {\bf n-1}}  \Int(R; S \cup r_{{\bf n-1} -T}) \rho_T \right) \rho_n  \\
& = \bigoplus_{n \notin T \subseteq \nn}  \Int(R; S  \cup r_{\nn -T}) \rho_T   + \bigoplus_{n \in T \subseteq \nn}  \Int(R; S \cup r_{\nn -T}) \rho_T  \\ 
& =  \bigoplus_{T \subseteq \nn}  \Int(R; S  \cup r_{\nn -T}) \rho_T.
\end{align*}
This completes the proof.
\end{proof}



It is natural to define  $\Int(R; S)$  for any multisubset $S$ of $R$ to be the intersection of the subrings $\Int(R;T)$ of $\Int(R)$ for all finite multisubsets $T$ of $S$.  The rings $\Int(R;S)$ have not been studied before, other than for $S \subseteq \{0,0,0,\ldots\}$.

We now provide alternative descriptions of $\Int(R[\rho])$ and $\Int(R; r)$ when $r$ is a non-zerodivisor of $R$. 
For any ring homomorphism $\varphi: R \longrightarrow S$, we let $$\Int(R; \varphi) = \{f \in \Int(R): \forall a,b \in R \ (\varphi(a) = \varphi(b) \Rightarrow \varphi(f(a) = \varphi(f(b))\}.$$ 
For any ideal $I$ of $R$, we let
$$\Int(R; I) = \{f \in \Int(R): \forall a,b \in R \ (a  \equiv b \, (\Mod I) \Rightarrow f(a)  \equiv f(b) \, (\Mod I) )\}.$$ 
Both of these rings are $R[X]$-subalgebras of $\Int(R)$.
Note that $\Int(R; \varphi)= \Int(R; \ker \varphi)$ and $\Int(R; I) = \Int(R; \pi)$, where $\pi: R \longrightarrow R/I$ is the natural projection.  Thus the two definitions are interchangeable.  Note also that $$I \Int(R) \subseteq \Int(R,I) \subseteq \Int(R; I),$$
where $\Int(R,I) = \{f \in \Int(R): f(R) \subseteq I\}$, and both $I \Int(R)$ and $\Int(R,I)$ are ideals of both of the rings $\Int(R)$ and $\Int(R; I)$.  Moreover, if $I$ is invertible, then an easy argument shows that $\Int(R,I) = I \Int(R)$.

\begin{lemma}
Let $R$ be a ring and $r$ a non-zerodivisor of $R$.  One has
$$\Int(R;rR) = \Int(R;r),$$
$$\Int(R,rR) = r\Int(R).$$
\end{lemma}

We let $R \times_A S = \{(r,s) \in R \times S: \varphi(r) = \psi(s)\}$ denote the pullback of a fixed diagram $R \stackrel{\varphi}{\longrightarrow} A \stackrel{\psi}{\longleftarrow}  S$ of ring homomorphisms.   We note the following.

\begin{lemma}\label{hyperboliciso}
Let $R$ be a ring and $r \in R$.  The map $$R[\rho] \longrightarrow R \times_{R/rR} R$$ acting by $x+y\rho \longmapsto (x,x+ry)$ is a surjective ring  homomorphism with kernel equal to $\operatorname{Ann}_R(r)\rho$.  In particular, if $r$ is a non-zerodivisor of $R$, then the map is an isomorphism with inverse acting by $(x,y) \longmapsto  x+ \frac{y-x}{r} \rho$.
\end{lemma}

\begin{proposition}\label{fibercor}
Let $R$ be a ring and $r \in R$.  The map $$\Int(R[\rho]) \longrightarrow \Int(R; r) \times_{\Int(R; r)/r\Int(R)} \Int(R; r)$$ acting by $f + g\rho \longmapsto (f,f+rg)$ is a surjective ring isomorphism with kernel $\operatorname{Ann}_{\Int(R)}(r)\rho$.   In particular, if $r$ is a non-zerodivisor of $R$, then the map is an isomorphism with inverse acting by $(f,g) \longmapsto f + \frac{g-f}{r}\rho$. 
\end{proposition}

\begin{example}
If $r$ is a unit of $R$, then $R[\rho] \longrightarrow R \times_{R/rR} R = R \times R$ is an isomorphism,  one has $\Int(R; rR) = \Int(R; R) = \Int(R)$, and therefore $\Int(R[\rho]) = \Int(R) + \Int(R)\rho  = \Int(R)[\rho] \cong \Int(R) \times \Int(R) = \Int(R) \times_{\Int(R)/r\Int(R)} \Int(R)$.
\end{example}

Let $I$ be an ideal of $R$, and let $C$ be a complete system of representatives in $R$ of the ring $R/I$.  Then the map $$\Phi: \Int(R; I) \longrightarrow \prod_{c \in C} R/I$$ acting by $f \longmapsto (f(c) \Mod I)_{c \in C}$ has kernel $\Int(R,I)$ and therefore one has an inclusion
$$\Int(R; I)/\Int(R,I) \longrightarrow \prod_{c \in C} R/I.$$ (The inclusion is an isomorphism if $R$ is a Dedekind domain with finite residue fields and $I$ is nonzero.)  Let us denote the image of the map $\Phi$ by $\Int(R:I)$, so that $\Int(R; I)/\Int(R,I) \cong \Int(R:I)$.  For each $\cc \in \Int(R:I)$, let $f_\cc \in \Int(R;I)$ with $\Phi(f_\cc) = \cc$.  Then $ \{f_\cc: \cc \in \Int(R:I)\}$ is a system of representatives in $\Int(R; I)$ of $\Int(R;I)/\Int(R,I)$, and therefore one has
$$\Int(R;I) = \coprod_{\cc \in \Int(R:I)} (f_\cc + \Int(R,I)) = \{f_\cc: \cc \in \Int(R:I)\}+\Int(R,I).$$
Therefore, if $r$ is a non-zerodivisor of $R$, then by Proposition \ref{hyperbolicint} one has
\begin{align*}
\Int(R[\rho])  & = \Int(R; rR) + \rho\Int(R)   \\
& = \{f_\cc: \cc \in \Int(R:rR)\}+\Int(R,rR) + \rho\Int(R)   \\
& =  \{f_\cc: \cc \in \Int(R:rR)\}+(r,\rho)\Int(R),
\end{align*}
and by Proposition \ref{fibercor} one has an isomorphism
$$\Int(R[\rho]) \longrightarrow \Int(R; rR) \times_{\Int(R: rR)} \Int(R; rR).$$
Note the special case where $\{f_\cc: \cc \in \Int(R:rR)\} \subseteq R[X]$, from which it follows that
$$\Int(R; rR) = R[X] + r\Int(R)$$ and
$$\Int(R[\rho]) = R[X] + (r,\rho)\Int(R).$$

\begin{example}
The map
$$\Phi: \Int(\ZZ; 2\ZZ) \longrightarrow \ZZ/2\ZZ \times \ZZ/2\ZZ$$
acting by $f \longmapsto (f(0),f(1))$ is surjective, and one has
$$\Int(\ZZ; 2\ZZ) =  \{0,1,X,1+X\}+2 \Int(\ZZ) = \ZZ[X]+2\Int(\ZZ).$$
Therefore, one has
\begin{align*}
\Int(\ZZ[j])  & = \{0,1,X,1+X\} + (2,1+j)\Int(\ZZ) \\
& = \ZZ[X] + (2,1+j)\Int(\ZZ).
\end{align*}
\end{example}

\begin{example}
Let $p$ be a prime number and $\ZZ[\rho] = \ZZ[T]/(T(T-p))$, where $\rho$ is the image of $T$.  The map
$$\Phi: \Int(\ZZ; p\ZZ) \longrightarrow (\ZZ/p\ZZ)^p$$
acting by $f \longmapsto (f(0), f(1), \ldots, f(p-1))$ is surjective and in fact is bijective when restricted to the set $$\{a_0+a_1X + a_2X(X-1) + \cdots + a_{p-1} X(X-1)\cdots (X-p+1): a_i \in \{0,1,\ldots,p-1\}\} \subseteq \ZZ[X].$$
Therefore, since $(k!,p) = 1$ for all $0 \leq k < p$, one has
$$\Int(\ZZ; p\ZZ) = \ZZ[X]+ p \Int(\ZZ) = \bigoplus_{k = 0}^p  \ZZ {X \choose k} \oplus \bigoplus_{k = p+1}^\infty p\ZZ {X \choose k}$$
and 
$$\Int(\ZZ[\rho]) = \ZZ[X] + (p, \rho)\Int(\ZZ).$$
\end{example}

\begin{example}
One has $2{X \choose 4}+{X \choose 2}, 2{X \choose 5}+{X \choose 3} \in \Int(\ZZ; 4\ZZ)$, and therefore $$\Int(\ZZ; 4\ZZ) \supsetneq \ZZ[X] + 4 \Int(\ZZ).$$  It seems reasonable to conjecture that  $$\Int(\ZZ; 4\ZZ)  =  \ZZ \oplus \ZZ X \oplus \bigoplus_{k = 2}^3 \ZZ 2 {X \choose  k} \oplus \bigoplus_{k = 4}^5 \ZZ\left(2{X \choose k}+{X \choose k-2} \right)  \oplus \bigoplus_{k = 6}^\infty \ZZ 4 {X \choose k}.$$
\end{example}

\begin{problem}
Compute $\Int(\ZZ;4)$, or more generally $\Int(\ZZ;n)$ for composite $n$.
\end{problem}

\section{Integer-valued polynomials over the idealization of a module}

Let $R$ be a ring and $M$ a (left) $R$-module.  One defines the {\it (Nagata) idealization of $M$ over $R$} to be the $R$-algebra $R(+)M$ whose $R$-module structure is the $R$-module $R \oplus M$ and whose multiplication is given by $$(r,m)(s,n) = (rs, rn+sm)$$ for all $(r,m), (s, n) \in R \oplus M$.  (It is defined alternatively as the quotient of the symmetric algebra $S(M) = \bigoplus_{k= 0}^\infty S^k(M)$ of $M$ over $R$ by the ideal $\bigoplus_{k= 2}^\infty S^k(M)$.)  It is called the ``idealization'' of $M$ because the $R$-module $M$ becomes an ideal $0(+)M$ of the idealization $R(+)M$.   

Idealization relates to Proposition \ref{dualintegervaluedpolynomials} in two ways.  First, note that $R[\varepsilon]$ is isomorphic to the idealization $R(+)R$ over $R$ of the free rank one $R$-module $R$.   Indeed, one can easily see that the map $R(+)R \longrightarrow R[\varepsilon]$ acting by $(a,b) \longmapsto a+b\varepsilon$ is an isomorphism of $R$-algebras.  Second, note  that, since $\Int^{(1)}(R) \subseteq \Int(R)$, we may consider $\Int(R)$ as an $\Int^{(1)}(R)$-algebra, hence as an $\Int^{(1)}(R)$-module, and then we have an $\Int^{(1)}(R)$-algebra isomorphism
$$\Int^{(1)}(R) (+) \Int(R) \longrightarrow \Int^{(1)}(R) + \Int(R)\varepsilon = \Int(R[\varepsilon])$$
acting by $(f,g) \longmapsto f + g \varepsilon$.  Thus, Proposition  \ref{dualintegervaluedpolynomials} implies the following.

\begin{corollary}\label{dualintegervaluedpolynomialscor}
One has canonical isomorphisms
$$\Int(R[\varepsilon]) \cong \Int(R(+)R) \cong \Int^{(1)}(R) \, (+) \Int(R)$$
for any ring $R$.
\end{corollary}

This motivates the following problem.

\begin{problem}
Let $R$ be a ring and $M$ an $R$-module.  Describe $\Int^{(k)}(R(+)M)$ for all nonnegative integers $k$.
\end{problem}

 Let $M$ be an $R$-module.   An element $a$ of $R$ is said to be a {\it non-zerodivisor on $M$} if $am = 0$ implies $m = 0$ for all $m \in M$.    Let $T(M)$ denote the localization $U^{-1}M$, where $U$ is the multiplicative subset $M^{\reg}$ of $R$ consisting of all non-zerodivisors on $M$.  Note that the canonical $R$-module homomorphism $M \longrightarrow T(M)$ is an inclusion, so we consider $M$ as an $R$-submodule of $T(M)$.  
 We let $M[X]  = \bigoplus_{k = 0}^\infty X^kM$ denote the set of all polyomials with coefficients in $M$, that is, expressions of the form $\sum_{k = 0}^n X^k m_k$ with $n$ a non-negative integer and $m_k \in M$ for all $k$.  Note that the $R$-module $M[X]$ is naturally an $R[X]$-module under the $R[X]$-scalar multiplication defined by $$\left(\sum_{j = 0}^m f_jX^j\right)\left( \sum_{k = 0}^n X^km_k\right) = \sum_{i = 0}^{m+n} X^i \left(\sum_{j+k = i} f_j m_k\right)$$
for all $f_j \in R$ and $m_k \in M$.   Observe that every $h = \sum_{k = 0}^n X^k m_k \in M[X]$ determines a function $R \longrightarrow M$ acting by $x \longmapsto h(x) = \sum_{k = 0}^n x^k m_k$.   We let
$$\Int(R,M) = \{h \in T(M)[X]: h(R) \subseteq M\},$$
which we call the {\it module of integer-valued polynomials on $M$ over $R$}.  We also let 
$$\Int_M(R) = \{f \in U^{-1}R[X]: f(R)M \subseteq M\},$$
which we call the {\it ring of integer-valued polynomials on $R$ over $M$}.  Note that $\Int_M(R)$ is an $R[X]$-subalgebra of $U^{-1}R[X]$.  Moreover, $\Int(R,M)$ is a module over the ring $\Int_M(R)$, since $T(M) = U^{-1}M[X]$ is a module over the ring $U^{-1}R[X]$ and if $f \in \Int_M(R)$ and $h \in \Int(R,M)$ then $fh \in \Int(R,M)$.   Note also that
$$M[X] \subseteq \Int(R,M) \subseteq T(M)[X]$$
and
$$(R \cdot 1)[X] \subseteq \Int_M(R) \subseteq U^{-1}R[X],$$
where $R \cdot 1$ is the image of $R$ in $U^{-1}R$.

\begin{remark}
If $I$ is an ideal of $R$, then $M = I$ is also an $R$-module and the notation  $\Int(R,M)$ conflicts with the notation $\Int(R,I)$ introduced in Section 3.  We now use the notation in the former sense rather than the latter.
\end{remark}

For all $z = (x,m) \in R(+)M$ we write $z = x+m\varepsilon$.  Note, however, that there is no element ``$\varepsilon$'' of $R(+)M$ since there is no ``1'' in $M$.     If $f \in R[X]$ and $h \in M[X]$, then we write $(f,h) = f + h \varepsilon$ for the polynomial $\sum_{i = 0}^{< \infty} X^i   (f_i+h_i \varepsilon)$ of $(R(+)M)[X]$.     Note that $M[X]$ is an ideal of the $R[X]$-algebra $(R(+)M)[X]$ when identified with $(0(+)M)[X]$.    Let $k$ be a nonnegative integer.  We define $$\Int^{(k)}_M(R) = \{f \in (M^{\operatorname{reg}})^{-1}R[X]: f, f',f'', \ldots, f^{(k)} \in \Int_M(R)\},$$
which is an $R[X]$-subalgebra of $T(R)[X]$.  We also define
$$\Int^{(k)}(R,M) = \{f \in T(M)[X]: f, f', f'', \ldots, f^{(k)} \in \Int(R,M)\},$$
which is a module over the ring $\Int^{(l)}_M(R)$ for $l \geq k$.

The following two lemmas and proposition are straightforward to verify.
\begin{lemma}
Let $R$ be a ring and $M$ an $R$-module, and let $(x,m) \in R(+)M$.  Then $(x,m)$ is a non-zerodivisor of $R(+)M$ if and only if $x$ is a non-zerodivisor on $M$ and $\ann_R(x) \cap \ann_R(m) = (0)$.  Moreover, $(x,m)$ is a unit of $R(+)M$ if and only if $x$ is a unit of $R$.
\end{lemma}

\begin{proposition}
Let $R$ be a ring and $M$ an $R$-module such that every non-zerodivisor on $M$ is a non-zerodivisor of $R$.  Let $(x,m) \in R(+)M$.  Then $(x,m)$ is a non-zerodivisor of $R(+)M$ if and only if $x$ is a non-zerodivisor on $M$, and $(x,m)$ is a unit of $R(+)M$ if and only if $x$ is a unit of $R$.  Moreover, the total quotient ring $T(R(+)M)$ of $R(+)M$ is given by $U^{-1}R(+)U^{-1}M$, where $U$ is the multiplicative subset of $R$ consisting of all non-zerodivisors on $M$.
\end{proposition}

\begin{lemma}\label{idealizationpolynomialfunctions}
Let $R$ be a ring and $M$ an $R$-module.  Every $F \in (R(+)M)[X]$ can be written uniquely in the form $F = f + h\varepsilon$, where $f \in R[X]$ and $h \in M[X]$.  Moreover, one has
$$F(x + m \varepsilon) = f(x) + (f'(x)m+h(x)) \varepsilon$$
for all $x\in R$ and all $m \in M$.
\end{lemma}

\begin{theorem}
Let $R$ be a ring and $M$ an $R$-module such that every non-zerodivisor on $M$ is a non-zerodivisor of $R$, and let $k$ be a nonnegative integer.  Then one has 
$$\Int^{(k)}(R,R(+)M) = \Int_M^{(k)}(R) \cap \Int^{(k)}(R)  +\Int^{(k)}(R,M)\varepsilon$$
and
$$\Int^{(k)}(R(+)M) = \Int^{(k+1)}_M(R)\cap \Int^{(k)}(R) + \Int^{(k)}(R,M)\varepsilon,$$
where also
$$\Int_M^{(k)}(R) \cap \Int^{(k)}(R) = (M^{\reg})^{-1}R[X] \cap \Int^{(k)}(R)$$
and
$$\Int_M^{(k+1)}(R) \cap \Int^{(k)}(R) = \{f \in \Int^{(k)}(R): f^{(k+1)}(R)M \subseteq M\}.$$  Moreover, one has canonical isomorphisms
$$\Int^{(k)}(R,R(+)M) = (\Int^{(k)}_M(R) \cap \Int^{(k)}(R)) \, (+) \Int^{(k)}(R,M)$$
and
$$\Int^{(k)}(R(+)M) \cong (\Int^{(k+1)}_M(R) \cap \Int^{(k)}(R))\, (+) \Int^{(k)}(R,M).$$
\end{theorem}

\begin{proof}
We prove the theorem for $k = 0$.  The proof easily extends to arbitrary $k$.  The total quotient ring of $R(+)M$ is given by $S = U^{-1}R(+)U^{-1}M$, where $U = M^{\reg}  \subseteq R^{\reg}$.   .  We may write any $F \in S[X]$ uniquely in the form $F = f+h \varepsilon$, where $f \in U^{-1}R[X]$ and $h \in U^{-1}M[X]$.    The equalities $$\Int(R,R(+)M) = \Int_M(R) \cap \Int(R)  +\Int(R,M)\varepsilon$$ and $$\Int_M(R) \cap \Int(R) = U^{-1}R[X] \cap \Int(R)$$ are then clear from the definitions and the equation $F(x) = f(x) + h(x)\varepsilon$Then \begin{align}\label{eq1} F(x+m\varepsilon) = f(x) + (f'(x)m+h(x)) \varepsilon \end{align}
for all $x\in R$ and all $m \in M$.  Suppose $F \in \Int(R(+)M)$.  Then $h \in \Int(R,M)$, so $f = f + 0 \varepsilon = F-h\varepsilon \in \Int(R(+)M)$.   Since $f \in \Int_M(R) \cap \Int(R)$, one  also has $f(x+ m\varepsilon) = f(x) + f'(x)m \varepsilon \in R(+)M$, and thus
$f'(x)m \in M$, for all $x \in R$ and $m \in M$, and therefore $f \in \Int_M^{(1)}(R) \cap \Int(R)$.  Conversely, if $f \in \Int_M^{(1)}(R) \cap \Int(R)$ and $h \in \Int(R,M)$, then from Eq.\ \ref{eq1} we see that $F = f+h\varepsilon \in \Int(R(+)M)$.
\end{proof}

Consider, for example, the case where $M = R^n$ is free of rank $n$ for some positive integer $n$.  One has $$R(+)R^n  \cong R[T_1, \ldots, T_n]/(T_1,\ldots,T_n)^2,$$
$$\Int(R,R^n) \cong \Int(R)^n,$$
$$\Int_{R^n}(R) = \Int(R),$$
$$\Int^{(1)}_{R^n}(R) = \Int^{(1)}(R).$$
We have 
$$R[\delta_1,\ldots, \delta_n] = R[T_1, \ldots, T_n]/(T_1,\ldots,T_n)^2,$$
where $\delta_i$ is the image of $T_i$ in the given quotient ring for all $i$.  This ring is called the ring of {\it $(n+1)$-dimensional dual numbers over $R$}  and is free of rank $n+1$ as an $R$-module, canonically isomorphic to $R \oplus \bigoplus_{i = 1}^n R\delta_i$.  Thus we have the following.

\begin{corollary}
Let $R$ be a ring and $n$ a positive integer.  Then $R[\delta_1,\ldots, \delta_n] \cong R (+)R^n$, and one has
$$\Int(R, R[\delta_1,\ldots, \delta_n]) = \Int(R)[\delta_1,\ldots,\delta_n]$$ and
$$\Int(R[\delta_1,\ldots, \delta_n])  = \Int^{(1)}(R) +(\delta_1,\ldots,\delta_n) \Int(R)[\delta_1,\ldots,\delta_n]  \cong \Int^{(1)}(R)\, (+) \Int(R)^n.$$
Equivalently, as  $R$-modules one has
$$\Int(R[\delta_1,\ldots, \delta_n])  = \Int^{(1)}(R) \oplus \bigoplus_{i = 1}^n \Int(R) \delta_i.$$
\end{corollary}

Another example is the case where $M = R/I$ is the quotient of $R$ by an ideal $I$.   Let $U$ denote the set of all non-zerodivisors mod $I$ of $R$.  (Thus, if $I = \ppp$ is prime, then $U = R-\ppp$.)  One has
 $$R\, (+)\, R/I  \cong R[\varepsilon]/I\varepsilon \cong R[X]/((X^2)+IX),$$
$$\Int(R,R/I) = \Int(R/I),$$
and, if every non-zerodivisor mod $I$ is a non-zerodivisor, then
$$\Int_{R/I}(R) = U^{-1}R[X] \cap \Int(R),$$
$$\Int^{(1)}_{R/I}(R) = U^{-1}(R)[X] \cap \Int^{(1)}(R).$$
Thus we have the following.

\begin{corollary}
Let $R$ be a ring and $I$ an ideal of  $R$ such that every non-zerodivisor mod $I$ of $R$ is a non-zerodivisor of $R$ (which holds, for example, if $I$ is a proper ideal of $R$ containing all zerodivisors of $R$).  Let $S =  R[X]/((X^2)+IX)$, and let $U$ be the set of all non-zerodivisors mod $I$ of $R$.  One has $S \cong R (+) R/I$ and
$$\Int(S) \cong (U^{-1}R[X] \cap \Int^{(1)}(R))\, (+) \, \Int(R/I).$$
\end{corollary}

\begin{corollary}\label{helppp}
Let $R$ be a ring and $\ppp$ a prime ideal of $R$ containing all zerodivisors of $R$.  Let $S =  R[X]/((X^2)+\ppp X)$.  One has $S \cong R (+) R/\ppp$ and
$$\Int(S) \cong (R_\ppp[X] \cap \Int^{(1)}(R)) \, (+) \, \Int(R/\ppp).$$
\end{corollary}

\begin{example}
One has $\Int(\ZZ[T]/(T^2,pT)) \cong (\ZZ_{(p)}[X] \cap \Int^{(1)}(\ZZ)) \, (+) \,  \FF_p[X]$ for any prime number $p$.
\end{example}

Another example is where $M = T(R)$ is the total quotient ring of $R$.

\begin{corollary}
Let $R$ be a ring.  Let $S =  R + T(R) \varepsilon = R+\varepsilon T(R)[\varepsilon] \subseteq T(R)[\varepsilon]$.  One has $S \cong R (+) T(R)$ and
$$\Int(S) = \Int(R) +  T(R)[X] \varepsilon = \Int(R) +  \varepsilon T(R)[X][\varepsilon] \cong \Int(R) \, (+) \, T(R)[X].$$
\end{corollary}

A consequence of the above corollary is that $\Int(R+\varepsilon T(R)[\varepsilon])$ has the structure of an $(R+\varepsilon T(R)[\varepsilon])$-plethory  if $\Int(R)$ has the structure of an $R$-plethory \cite[Proposition 7.29]{ell}.

The ordered monoid of regular fractional ideals of $R$ is isomorphic to the ordered monoid of regular fractional ideals of $R+\varepsilon T(R)[\varepsilon]$. It follows that the ring $R$  is Dedekind (resp., Krull, Pr\"ufer, PVMR) if and only if $R+\varepsilon T(R)[\varepsilon]$ is Dedekind (resp., Krull, Pr\"ufer, PVMR), all four of these notions being defined naturally for rings with zerodivisors \cite{huc}.  (For example, a ring $R$ is said to be {\it Dedekind} if every regular ideal of $R$ is invertible.)  In particular, if $D$ is a Krull domain, then $\Int(D)$ is a PVMD \cite{tar}, and therefore $\Int(D+\varepsilon T(D)[\varepsilon])$ is a PVMR.  Likewise, if $D$ is a Dedekind domain with all residue fields finite, then $\Int(D)$ is a Pr\"ufer domain, and therefore $\Int(D+\varepsilon T(D)[\varepsilon])$ is a Pr\"ufer ring.

\begin{conjecture}
If $R$ is a Krull ring, then $\Int(R)$ is a PVMR and has the structure of an $R$-plethory.
\end{conjecture}

\section{Polynomial torsion}

Let $R$ be a ring.  An $R$-module $M$ is said to be {\it without polynomial torsion} if $f(R) = (0)$ implies $f = 0$ for all $f \in M[X]$ \cite{cah0}. Clearly every $R$-module without polynomial torsion is $R$-torsion-free.

Observe that there is a canonical $R$-module homomorphism $\operatorname{eval}: M[X] \longrightarrow M^R$, where $M^R$ is the $R$-module of all functions from $R$ to $M$, or, equivalently, $M^R$ is the direct product $\prod_{r \in R} M$.  The image of eval is the $R[X]$-module $P(R,M)$ of all {\it polynomial functions from $R$ to $M$}.  The kernel  of eval is the $R[X]$-submodule $I_M = \{f \in M[X]: f(R) = (0)\}$ of all polynomials in $M[X]$ vanishing on $R$. Thus one has $P(R,M) \cong M[X]/I_M$ as $R[X]$-modules.   Moreover, the $R$-module $M$ is without polynomial torsion if and only if $I_M = (0)$, if and only if the (surjective) $R[X]$-module homomorphism $\operatorname{eval}: M[X] \longrightarrow P(R,M)$ is an isomorphism.

Note that if $S = M$ is an $R$-algebra, then $P(R,S)$ is an $R[X]$-algebra, $I_S$ is an ideal of the $R[X]$-algebra $S[X]$, and the isomorphism $P(R,S) \cong S[X]/I_S$ is an isomorphism of $R[X]$-algebras.

\begin{example}[{\cite[Lemma 1.1]{nar}}] Let $D$ be an integral domain.
\begin{enumerate}
\item $D$ is without polynomial torsion if and only if $D$ is not a finite field.
\item If $D = \FF_q$ is a finite field, then $I_D = (X^q-X)$.
\end{enumerate}
\end{example}

\begin{example}[{\cite[Corollary 5]{gil}}]
 If $R$ is a finite ring, then $I_R$ is principal if and only if $R$ is reduced, if and only if $R$ is a finite direct product of finite fields.
\end{example}

\begin{example}[{\cite[Chapter II Exercise 3]{nar}}]  Let $R$ be a ring.
\begin{enumerate}
\item An $R$-submodule of an $R$-module without polynomial torsion is without polynomial torsion.
\item The direct sum and direct product of $R$-modules without polynomial torsion  is without polynomial torsion.
\item If $M$ is an $R$-module and $N$ is an $R$-submodule of $M$ such that both $N$ and $M/N$ are without polynomial torsion, then $M$ is without polynomial torsion.
\end{enumerate}
\end{example}

\begin{proposition}
Let $R$ be a ring and $M$ an $R$-module.  Then $\Int(R,M) = M[X]$ if and only if the $R$-module $T(M)/M$ is without polynomial torsion.  
\end{proposition}

\begin{proof}
Let $f \in T(M)[X]$, and let $\overline{f}$ denote the image of $f$ in $(T(M)/M)[X]$.  Note that $f \in M[X]$ if and only if $\overline{f} = \overline{0}$ in $(T(M)/M)[X]$, while $f \in \Int(R,M)$ if and only if $\overline{f}(R) = (0)$ in $T(M)/M$.  The proposition follows.
\end{proof}

\begin{corollary}
Let $R$ be a ring.  Then $\Int(R) = R[X]$ if and only if the $R$-module $T(R)/R$ is without polynomial torsion.  
\end{corollary}

Let $R$ be a ring and $M$ an $R$-module.  An {\it associated prime} of $M$ is a prime ideal of $R$ of the form $\ann_R(m)$ for some $m \in M$, where $\ann_R(m) = (0 :_R m) = \{r \in R: rm = 0\}$ is the {\it annihilator of $m$}.  

If $x = m/u  \in T(M)$, where $m \in M$ and $u \in M^{\reg}$, and if $\overline{x}$ is the image of $x$ in $T(M)/M$, then $\ann_R(\overline{x}) = (M :_R m/u) = (uM :_R m)  = \{r \in R: rm \in uM\}$.  An ideal of $R$ of the form $ (uM :_R m)$ for $m \in M$ and $u \in M^{\reg}$ is a {\it conductor ideal of $R$ in $M$}.  In particular, the associated primes of $T(M)/M$ are the prime conductor ideals of $R$ in $M$.

\begin{proposition}[{\cite[Corollaire 1]{cah0}}]
Let $R$ be a Noetherian ring and $M$ an $R$-module.  Then $M$ is without polynomial torsion if and only if $R/\ppp$ is finite for every associated prime $\ppp$ of $M$.
\end{proposition}

\begin{corollary}
Let $R$ be a Noetherian ring and $M$ an $R$-module.  Then $\Int(R,M) = M[X]$ if and only if $R/\ppp$ is finite for every associated prime $\ppp$ of the $R$-module $T(M)/M$, if and only if $R/\ppp$ is finite for every prime conductor ideal $\ppp$ of $R$ in $M$.
\end{corollary}

\end{document}